\documentclass[final]{dmtcs-episciences}

\usepackage[utf8]{inputenc}
\usepackage{subfigure}

\usepackage[round]{natbib}

\usepackage{amssymb}%
\usepackage{amsmath}%
\usepackage{amsthm}%
\usepackage{hyperref}%
\usepackage{graphicx}%
\usepackage{xcolor}%
\usepackage{mathrsfs}

\allowdisplaybreaks%

{\theoremstyle{plain}%
 \newtheorem{theorem}{Theorem}
 
 \newtheorem{lemma}{Lemma}%
}
{\theoremstyle{remark}

}
{\theoremstyle{definition}
\newtheorem{definition}{Definition}
\newtheorem{example}{Example}
}

\author[John M. Campbell and Narad Rampersad]{John M. Campbell\affiliationmark{1}
  \and Narad Rampersad\affiliationmark{2}\thanks{The work of Narad Rampersad is supported by the Natural
  Sciences and Engineering Research Council of Canada (NSERC), [funding reference number RGPIN-2025-04076]}}

\title[Subword enumeration up to stack-sorting equivalence]{Subword enumeration up to stack-sorting equivalence}

\affiliation{
  Dept. of Mathematics and Statistics, Dalhousie University, Halifax, Canada\\
  Dept. of Mathematics and Statistics, University of Winnipeg, Winnipeg, Canada}

\keywords{subword, complexity function, permutation, stack sorting, partial sorting algorithm, 
 West's stack-sorting map, automatic sequence, paperfolding sequence, Thue--Morse sequence, {\tt Walnut}, $k$-abelian complexity}

\begin{document}
\publicationdata{vol. 28:3}{2026}{27}{10.46298/dmtcs.18117}{2026-04-29; 2026-04-29; 2026-10-07; 2026-11-19}{2026-11-20}
\maketitle
\begin{abstract}
 Defant and Kravitz introduced generalizations of West's stack-sorting map $s$ from permutations to finite words. This raises questions 
 as to how such generalizations could be applied in the field of combinatorics on words. The Defant--Kravitz generalizations of $s$
 depend on how repeated occurrences of the same character within a word may be repositioned, according to their $\textsf{tortoise}$
 and $\textsf{hare}$ operations. As demonstrated in this paper, these operations provide a natural way of extending abelian complexity 
 functions for infinite sequences, in a way that gives light to structural properties associated with infinite words. We apply these new
 ideas to two famous infinite words: the paperfolding word and the Thue--Morse word. In the case of the Thue--Morse word, we discover
 an interesting connection to the previous work of several authors, such as de Luca and Varricchio, on the ``special'' factors of the
 Thue--Morse word. This may be seen as providing a basis for a new and interdisciplinary area linking the combinatorics about the 
 stack-sorting of permutations with the field of combinatorics on words.
\end{abstract}

\section{Introduction}
 What is referred to as \emph{West's stack-sorting map} traces back to Knuth's classic texts on computer programming 
 \cite[\S2.2.1]{Knuth1973} and was given, in an equivalent way, in 1990 in the seminal PhD thesis of \cite{West1990}. Since partial sorting 
 algorithms on permutations helped to lead toward areas in combinatorics devoted to permutation patterns, this, in turn, motivates the 
 extension of such algorithms to combinatorial objects generalizing permutations. In this paper, we show how generalizations due 
 to \cite{DefantKravitz2020} of West's stack-sorting map can be applied in the study of structural and enumerative 
 properties of automatic sequences and in the context of combinatorics on words. 

For any finite word $\sigma$ consisting of distinct letters from $\{1,2,\ldots\}$, let $\sigma$ be written as
\begin{equation}\label{sigmaLnR}
 \sigma = L n R
\end{equation}
where $n$ is the largest letter appearing in $\sigma$ and $L$ and $R$ are possibly empty subwords.
West's stack-sorting map $s$ may be defined so that $s$ maps the empty word to 
the empty word and so that 
\begin{equation}\label{srec}
 s(\sigma) = s(L) s(R) n. 
\end{equation}
This definition naturally raises
 questions as to how the recurrence relation in \eqref{srec} could be generalized so as to define stack-sorting operations on words 
 with possibly repeated characters, and \cite{DefantKravitz2020} were the first to introduce and apply such 
 generalizations. In view of research works subsequent and related to the Defant--Kravitz generalizations of $s$ by
 \cite{Cerbai2021,Defant2020Counting,Defant2020Descents,Defant2020FertilityAustralas,Defant2020FertilityJ,Defant2020Polyurethane,Defant2019,DefantEngenMiller2020,DefantKravitz2024}, 
 it appears that these generalizations have not been applied directly in the study of automatic sequences or in terms of complexity 
 functions on infinite words/sequences. This forms the main purpose of our paper. 

 As in the work of \cite{DefantKravitz2020}, we express how the extension of $s$ from permutations to words is 
 inspired by many past research efforts on the extension of the study of pattern avoidance in permutations to pattern avoidance in words, 
 together with past research on sorting procedures on words. Ideally, through new applications of and investigations devoted to the 
 application of sorting procedures on finite subwords of automatic sequences, this could lead toward a deeper understanding 
 of aspects about such subwords in terms of, say, growth properties when it comes to the enumeration of subwords up to a given point 
 in an automatic sequence, or in terms of, for example,
 Morse--Hedlund-type results (cf.\ \cite{MorseHedlund1940}) characterizing eventual 
 periodicity or periodicity-like conditions in terms of the behavior of complexity functions. 

 For a given sequence $\text{{\bf x}}$, it is of interest in its own right how equal-length subwords of $\text{{\bf x}}$ relate to one another
 in terms of partial sortings of or repositionings of characters. The notion of stack-sorting equivalence on subwords introduced in this 
 paper helps to formalize this idea and provides applications based on this idea to the automatic sequences known as the paperfolding 
 sequence and the Thue--Morse sequence. 

\section{Background}
 For the empty word $\varepsilon$, set $\textsf{hare}(\varepsilon) = \textsf{tortoise}(\varepsilon) = \varepsilon$, as in the work of \cite{DefantKravitz2020}. For a nonempty word $w$, let $n$ be the largest character appearing in $w$, 
 and suppose that this maximal character appears $k$ times, writing 
\begin{equation}\label{wAnAn}
 w = A_1 n A_2 n \cdots n A_{k+1}
\end{equation}
 by direct analogy with \eqref{sigmaLnR}, and with the understanding that every character appearing in $A_i$ (for $i \in \{ 1, 2, \ldots, k + 
 1 \}$) does not exceed $n-1$, and with the understanding that $A_i$ may or may not be $\varepsilon$. 
 Defant and Kravitz then defined the operator $\textsf{hare}$ recursively so that 
\begin{equation}\label{harew}
\textsf{hare}(w) = \textsf{hare}(A_{1}) \, \textsf{hare}(A_2) 
 \, \cdots \, \textsf{hare}(A_{k+1}) \, \underbrace{nn\cdots{n}}_{k}. 
\end{equation}

\begin{example}
 Define the word $w$ in $\{ 0, 1, \ldots, 9 \}^{\ast}$ so that 
\begin{equation}\label{piw}
 w = 314159265358979323846264. 
\end{equation}
 We then have that 
 $$ \textsf{hare}(w) = \textsf{hare}(31415) \, 
 \textsf{hare}(265358) \, \textsf{hare}(7) \, \textsf{hare}(323846264) \, 999, $$
 and the iterative application of \eqref{harew} gives us 
\begin{equation}\label{finalpihare}
 \textsf{hare}(w) = 131452355687233424668999. 
\end{equation}
\end{example}

 Informally, and in contrast to the recursion for $s$ based on the decomposition in \eqref{sigmaLnR}, there is a matter of ambiguity when 
 it comes to how the $n$-characters in \eqref{wAnAn} should be repositioned, 
 and/or which such $n$-characters should be repositioned, e.g., in the context of a given application. 
 This leads us to Defant and Kravitz's alternative to the $\textsf{hare}$ operator, writing 
\begin{multline*}
 \textsf{tortoise}(w) = \textsf{tortoise}(A_1) \, \textsf{tortoise}(A_2) \, n \, \textsf{tortoise}(A_3) 
 \, n \, \cdots \\ \, n \, \textsf{tortoise}(A_k) \, n \, \textsf{tortoise}(A_{k+1}) \, n. 
\end{multline*}
 The terminology associated with the mappings
 $\textsf{hare}$ and $\textsf{tortoise}$ may be thought of as, informally, referring to 
 the ``speeds'' of the respective mappings in terms of 
 obtaining a lexicographically sorted word
 via iterative applications. 

\begin{example}
 For the word $w$ in \eqref{piw}, we find that 
$$ \textsf{tortoise}(w) = 131452355689792334264689. $$
\end{example}

 If $b$ is a binary word, then 
\begin{equation}\label{equivbin}
 \textsf{hare}(b) = \operatorname{sort}(b), 
\end{equation}
 for the lexicographically sorted version $\operatorname{sort}(b)$ of $b$. So, for an infinite and binary sequence $\text{{\bf x}}$, if 
 we enumerate subwords of $\text{{\bf x}}$ of a given length $k$ and up to the equivalence relation
 whereby two finite words are equivalent if the images of these words under 
 $\textsf{hare}$ are equal, 
 then this is equivalent to enumerating length-$k$ subwords of $\text{{\bf x}}$ up to equivalence by permutations of finite 
 subwords. This illustrates how the Defant--Kravitz operators $\textsf{hare}$ and $\textsf{tortoise}$ 
 provide natural generalizations 
 of the concept of an \emph{abelian complexity function} reviewed below 
 and could form a bridge between 
 the study of stack sorting and the field of \emph{abelian combinatorics on words} reviewed below. 
 It seems that this connection has not been considered previously, with reference to the \cite{FiciPuzynina2023}
 survey on abelian combinatorics on words, 
 and again with reference to 
 extant research works related to Defant and Kravitz's research on stack sorting on words by
 \cite{Cerbai2021,Defant2020Counting,Defant2020Descents,Defant2020FertilityAustralas,Defant2020FertilityJ,Defant2020Polyurethane,Defant2019,DefantEngenMiller2020,DefantKravitz2024}. 

\subsection{Abelian combinatorics on words}
 In this paper we use the notation $\mathbb{N}$ and $\mathbb{N}_0$ as follows:
 $\mathbb{N} = \{1,2,\ldots\}$ and $\mathbb{N}_0 = \mathbb{N} \cup \{0\}$.
 Let $\text{{\bf x}}$ denote an infinite sequence over a finite alphabet.
 The associated \emph{factor complexity} or \emph{subword 
 complexity function}
\begin{equation}\label{displaycomplexity}
 \rho_{\text{{\bf x}}}\colon \mathbb{N}_{0} \to \mathbb{N}
\end{equation}
 maps a domain element $n \in \mathbb{N}_{0}$ to the number of length-$n$ factors or subwords appearing within $\text{{\bf x}}$ 
 (when written as an infinite word). 
 Naturally, the enumeration of subwords provides a central part of combinatorics on words. 
 This points toward the significance of the roles played by natural variants of complexity functions as in 
 \eqref{displaycomplexity}. In this regard, 
 counting subwords up to rearrangements of characters
 is at the core of abelian combinatorics on words, 
 again with reference to the \cite{FiciPuzynina2023} survey on this discipline. 

 The \emph{abelian complexity function}
\begin{equation}\label{abcomdisplay}
 \rho_{\text{{\bf x}}}^{\operatorname{ab}}\colon \mathbb{N}_{0} \to \mathbb{N} 
\end{equation}
 is defined to map $n \in \mathbb{N}_{0}$ to the number, up to equivalence by (possibly trivial) permutations on finite subwords, of 
 length-$n$ factors of $\text{{\bf x}}$. Complexity functions as in \eqref{abcomdisplay} may be seen as providing a prototypical way of 
 enumerating fixed-length subwords up to what is typically meant by abelian equivalence (cf.\ \cite{FiciPuzynina2023}). However, 
 complexity functions as in \eqref{abcomdisplay} scratch the surface when it comes to the entire field of abelian combinatorics, 
 with regard to Fici and Puzynina's exposition on abelian powers, abelian critical exponents, abelian square factors, abelian antipowers, 
 generalizations of abelian powers, abelian pattern avoidance, abelian periodicity theorems, abelian primitive words, abelian borders, 
 $k$-abelian equivalence, weak abelian equivalence, and many further topics (see \cite{FiciPuzynina2023}). Since stack-sorting procedures on 
 words are not considered in Fici and Puzynina's survey, this suggests that generalizing abelian complexity functions via 
 the Defant--Kravitz $\textsf{tortoise}$ and $\textsf{hare}$ operations has not been considered
 in any equivalent way before. 

\subsection{The complexity of automatic sequences}
 For a full background on automatic sequences, we cite the standard text on automatic sequences by \cite{AlloucheShallit2003}. Informally, 
 the inherently recursive nature of automatic sequences can be thought of as reflecting how the study of subwords of such sequences 
 helps to give light to the nature of such sequences and the applications of such sequences within number theory and computer science. 
 This is reflected in a number of recent, notable research publications concerning complexity functions on automatic sequences, as in 
 the work of \cite{GolafshanRigoWhiteland2026} on the $k$-binomial complexity for generalized Thue--Morse sequences, 
 the work of \cite{LuChenWenWu2024} on the 2-binomial complexity of generalized Thue--Morse sequences, 
 the work of \cite{Shallit2023Proof} on the cyclic complexity of the Thue--Morse sequence, 
 the work of \cite{Rampersad2023} on the periodic complexity of the Thue--Morse sequence, the Rudin--Shapiro sequence, 
 and the period-doubling sequence, 
 and the work of \cite{LejeuneLeroyRigo2020} on the $k$-binomial complexity of the Thue--Morse sequence. 

 Two of the most basic and prototypical instances of non-periodic automatic sequences are the Thue--Morse sequence and the 
 paperfolding sequence. While it is well-known that the abelian complexity of the Thue--Morse sequence is $2$-periodic, 
 the abelian complexity function of the paperfolding sequence
 is non-periodic and $2$-regular, as proved by \cite{MadillRampersad2013}. 
 This motivates the material in Section \ref{secpaperstack} below
 on a new, abelian-like complexity function applied to the paperfolding sequence. 
 In Section \ref{secTMmain}, 
 we investigate this abelian-like complexity function 
 as applied to the TM sequence, uncovering an unexpected connection with the work 
 of \cite{deLucaVarricchio1989} 
 concerning distinguished factors of the TM sequence. 

 In this paper we only consider binary words. For such words, the equivalence in \eqref{equivbin} implies that two words have the same 
 image under $\textsf{hare}$ if and only if
they are abelian equivalent (i.e., they have the same number of occurrences of each letter).
We therefore restrict our attention in this paper to the $\textsf{tortoise}$ operation.

\section{Tortoise complexity}\label{sectTortcomp}
 Adopting notation from \cite{DefantKravitz2020}, 
 we let $\langle w \rangle_{\textsf{tortoise}}$ denote
	 the smallest nonnegative integer such that 
\begin{equation}\label{tortsort}
\textsf{tortoise}^{\left( \langle w \rangle_{\textsf{tortoise}} \right)}(w) = \operatorname{sort}(w). 
\end{equation}
 So, letting 
 $\text{{\bf x}}$ be an infinite sequence, we have that $\rho_{\text{{\bf x}}}^{\operatorname{ab}}(n)$ is equal to the number of 
 length-$n$ subwords of $\text{{\bf x}}$, up to the equivalence relation whereby two words $w$ and $v$ 
 are considered to be equivalent if 
 $\textsf{tortoise}^{\left( \langle w \rangle_{\textsf{tortoise}} \right)}(w) = \textsf{tortoise}^{\left( 
 \langle v \rangle_{\textsf{tortoise}} \right)}(v)$. 
 So, if we instead let, for a fixed parameter $k$, words 
 $w$ and $v$ be considered to be equivalent if 
\begin{equation}\label{tortoisek}
 \textsf{tortoise}^{(k)}(w) = \textsf{tortoise}^{(k)}(v),
\end{equation}
 then, informally, the corresponding complexity functions can be thought of as getting ``closer and closer'' to 
 $\rho_{\text{{\bf x}}}^{\operatorname{ab}}$ as $k$ grows larger and larger, in a way that can be thought of as capturing the extent 
 of ``abelianness,'' in a way that depends on $k$. 
 This may be seen as a natural variant of the notion of $k$-abelian equivalence, 
 which refers to $w$ and $v$ 
 starting with the same prefix of length $k-1$ and ending
 with the same suffix of length $k-1$, with the assumption that 
 the number of occurrences of $x$ in $w$ equals the 
 number of occurrences of $x$ in $v$ for all words $x \in A^{\ast}$
 such that $|x| = k$. 

 The notion of $k$-abelian complexity traces back to the work of \cite{Karhumaki1980} on generalized {P}arikh mappings. 
 In the work of \cite{ChenLuWu2018} concerning the $k$-abelian complexity of automatic sequences, 
 it is noted that $k$-abelian complexity is a measure of the disorder of a given sequence. Our notion of $k$-tortoise complexity 
 suggested in \eqref{tortoisek} and defined below measures the disorder of a given sequence in a different way and in what may be seen 
 as a more natural way. For the case of a binary word $w$ with at least $k$ characters equal to $1$, the application of 
 $\textsf{tortoise}^{(k)}$ to $w$ has the effect of taking the first $k$ occurrences of $1$-characters in $w$ and repositioning all $k$ of 
 these $1$-characters at the end of $w$. 
 This may be seen as a useful and natural operation in the context of abelian combinatorics on words, 
 in view of the many operations related to abelian complexity functions
 covered in the \cite{FiciPuzynina2023} survey. 
 
\begin{definition}\label{ktorteq}
 Let $k$ be a positive integer. Two words $w$ and $v$ (over a totally ordered alphabet) are \emph{$k$-tortoise-equivalent}, writing 
 $w \sim_{\textsf{t}^{(k)}} v$, if \eqref{tortoisek} holds. For the $k = 1$ case, we let $k$-tortoise equivalence be referred to as 
 \emph{tortoise equivalence}, and, in this case, we write $ \sim_{\textsf{t}} $ in place of $\sim_{\textsf{t}^{(1)}}$. 
\end{definition}

\begin{example}
 If we take the word
\begin{equation*}
 w = 001100011011100, 
\end{equation*}
 which is a subword of the paperfolding word defined in Section \ref{secpaperstack} below, 
 and define
 $$ v = 010100011011100, $$
 which is not a subword of the same paperfolding word, 
 then we find that 
 $w \sim_{\textsf{t}} v$, since 
 $$ \textsf{tortoise}(w) = \textsf{tortoise}(v) = 001000110111001. $$
\end{example}
 
 The complexity function in Definition \ref{definektort} provides a main object of investigation in this paper. 

\begin{definition}\label{definektort}
 Let $k$ be a positive integer, and let $\text{{\bf x}}$ be a sequence over a finite alphabet. 
 Define the \emph{$k$-tortoise complexity function}
\begin{equation}\label{displayrhotort}
 \rho_{\text{{\bf x}}}^{\textsf{t}^{(k)}}\colon \mathbb{N}_{0} \to \mathbb{N}
\end{equation}
 so that \eqref{displayrhotort} maps $n \in \mathbb{N}_{0}$ to the number, up to equivalence under $\sim_{\textsf{t}^{(k)}}$, of
 length-$n$ subwords of $\text{{\bf x}}$. For the $k = 1$ case, 
 we refer to $1$-tortoise complexity as \emph{tortoise complexity}, 
 and we rewrite $ \rho_{\text{{\bf x}}}^{\textsf{t}^{(1)}}$ as $ \rho_{\text{{\bf x}}}^{\textsf{t}}$.
\end{definition}

 For a given sequence $\text{{\bf x}}$ over a finite alphabet, and for a fixed parameter $k$, the integer sequences
\begin{equation}\label{compareseq}
 \left( \rho_{\text{{\bf x}}}^{\textsf{t}^{(k)}}(n) : n \in \mathbb{N}_{0} \right) 
 \ \ \ \text{and} \ \ \ \left( \rho_{\text{{\bf x}}}^{\operatorname{ab}}(n) : n \in \mathbb{N}_{0} \right) 
\end{equation}
 will either be equal or will agree up to a certain point. 
 Disregarding the degenerate case whereby the sequences in 
 \eqref{compareseq} agree completely or agree after a certain point
 (as would be the case if $\text{{\bf x}}$ is a constant sequence), 
 let $\operatorname{abel}_{\text{{\bf x}}}(k)$ denote the unique integer such that 
 $$ \forall n \leq \operatorname{abel}_{\text{{\bf x}}}(k) \ 
 \rho_{\text{{\bf x}}}^{\textsf{t}^{(k)}}(n) = \rho_{\text{{\bf x}}}^{\operatorname{ab}}(n) $$
 and such that 
\begin{equation}\label{guarantee}
 \rho_{\text{{\bf x}}}^{\textsf{t}^{(k)}}(\operatorname{abel}_{\text{{\bf x}}}(k)+1) 
 \neq \rho_{\text{{\bf x}}}^{\operatorname{ab}}(\operatorname{abel}_{\text{{\bf x}}}(k)+1). 
\end{equation}
 As suggested above, the statistic $\operatorname{abel}_{\text{{\bf x}}}(k)$ can be thought of as 
 quantifying ``how abelian'' a sequence is 
 relative to the effect of a $k$-fold application of $\textsf{tortoise}$. 

 Suppose that the sequences in \eqref{compareseq} do not agree after some point, for all positive integers $k$, and 
 again for a sequence $\text{{\bf x}}$ over a finite alphabet. 
 Then, for all $N \in \mathbb{N}$, it can be shown that there exists a value $K$ such that 
\begin{equation}\label{desireabel}
 \operatorname{abel}_{\text{{\bf x}}}(K) \geq N.
\end{equation} 
 This raises questions as to how the results of \cite{DefantKravitz2020} concerning the quantity $\langle w \rangle_{\textsf{tortoise}}$ 
  mentioned at the beginning of this section could be applied in the estimation of or evaluation of 
 expressions as in $\operatorname{abel}_{\text{{\bf x}}}(k)$. 
 For the purposes of this paper, 
 we investigate the problem 
 of evaluating tortoise complexity functions explicitly 
 on automatic sequences, 
 and we return to the problem of evaluating 
 $\operatorname{abel}_{\text{{\bf x}}}(k)$ in Sections \ref{secpaperstack} and \ref{secConclusion} below. 

\section{Tortoise complexity and left special factors}
A crucial tool in the study of subword complexity functions is the notion
of ``special'' subwords. If ${\bf x}$ is an infinite word then a subword $w$ of ${\bf x}$
is \emph{right special} (sometimes just \emph{special}) if there are distinct
letters of the alphabet $a$ and $b$ such that both $wa$ and $wb$ are subwords of
${\bf x}$. Similarly, if both $aw$ and $bw$ are subwords of ${\bf x}$, then
$w$ is \emph{left special}. Left special subwords are also relevant to the computation
of $\rho_{\bf x}^{\textsf{t}}(n)$ in the following way:

\begin{lemma}\label{lem_left_sp}
 Let $w$ and $w'$ be distinct equal-length subwords of an infinite binary word ${\bf x}$ such that
 $w \sim_{\emph{\textsf{t}}} w'$. Let $z$ be the longest common suffix of $w$ and $w'$. Then there exist
 $u \in 0^*$ and $v\in 0^+$ such that $\{w,w'\} = \{uv1z, u1vz\}$. In particular, $z$ is a left special
 factor of ${\bf x}$.
\end{lemma}

\begin{proof}
If $w = A_1 1 A_2 1 \cdots 1 A_{k+1}$ and $w' = A_1' 1 A_2' 1 \cdots 1 A_{k+1}'$,
where each $A_i, A_i' \in 0^*$, then $w \sim_{\emph{\textsf{t}}} w'$
implies that $A_1A_2 = A_1'A_2'$ and $A_i=A_i'$ for $i \geq 3$.
\end{proof}

 Consider an infinite binary word ${\bf x}$ where the $1$'s occur with \emph{bounded gaps}; i.e., there exists a constant $B$ such that 
 every factor of ${\bf x}$ of length $B$ contains a $1$.

\begin{theorem}
 Let ${\bf x}$ be an infinite binary word such that $\rho_{\bf x}(n) \in O(n)$ and the $1$'s in ${\bf x}$ occur with bounded gaps. 
 There exists a constant $C$ such that
$$ \rho_{\bf x}(n) - C \leq \rho_{\bf x}^{\emph{\textsf{t}}}(n) \leq \rho_{\bf x}(n) $$
for all sufficiently large $n$.
\end{theorem}

\begin{proof}
 Consider a $\sim_{\textsf{t}}$-equivalence class of ${\bf x}$ containing at least two distinct words $w$ and $w'$ of length $n$.
Observe that $w$ and $w'$ contain at least one $1$. Applying Lemma~\ref{lem_left_sp}, write $w=xy$ and $w'=x'y$, where $y$ is left 
 special and $x$ and $x'$ each
contain exactly one $1$. Let $|y|=m$. Note that since the $1$'s in ${\bf x}$
occur with bounded gaps, there is a bounded number of choices for $x$ and $x'$,
which implies that there is a bounded number of possibilities for $m$.

 Now, since ${\bf x}$ is an infinite binary word, note that the number of left special factors of length $m$ is equal to $\rho_{\bf x}(m+1) - \rho_{\bf x}(m)$ if the prefix of ${\bf x}$ 
 of length $m$ occurs at least twice in ${\bf x}$ and is equal to $\rho_{\bf x}(m+1) - \rho_{\bf x}(m) + 1$
otherwise. A deep result of \cite{Cassaigne1996} states that $\rho_{\bf x}(m) \in O(m)$
if and only if there is a constant $D$ such that $\rho_{\bf x}(m+1) - \rho_{\bf x}(m) \leq D$
for all $m$. Hence, not only is there a bounded number of possibilities for 
the length of the left special factor $y$, but for each possible length there are
at most $D+1$ choices for $y$.

 The conclusion is that there is a bounded number of $\sim_{\textsf{t}}$-equivalence classes of ${\bf x}$ that contain more than one 
 element, and furthermore, the size of such classes is bounded as well. This is sufficient to establish the claim.
\end{proof}

Much is known about the right special subwords of the Thue--Morse word ${\bf t}$.
It is well-known that if $w$ is a subword of ${\bf t}$ then so is the \emph{reversal}
of $w$, denoted $w^R$. It follows that if $w$ is a right special subword of ${\bf t}$
then $w^R$ is a left special subword of ${\bf t}$. De Luca and Varricchio 
gave a formula for the number of right special factors of ${\bf t}$
of length $n$. The same applies to left special factors by our earlier observation,
so from~\cite[Lemma~4.3]{deLucaVarricchio1989} we get the following fact.

\begin{lemma}\label{lem_dLV}
For any $n\geq 2$, the number of left special factors of ${\bf t}$
of length $n$ is
\[
\begin{cases}
4 &\text{if } n\leq 3\cdot 2^{\lfloor \log_2(n-1) \rfloor - 1},\\
2 &\text{if } n> 3\cdot 2^{\lfloor \log_2(n-1) \rfloor - 1}.
\end{cases}
\]
\end{lemma}

Later, we will see how this information allows us to easily compute $\rho_{\bf t}^{\textsf{t}}(n)$.

\section{Paperfolding and stack sorting}\label{secpaperstack}
 For a positive integer $n$, write $n = n'2^{k}$ for an odd number $n'$, and define 
 \[ f_{n} = \begin{cases} 
 0 & \text{if $n' \, \equiv \, 1(\operatorname{mod} \, 4)$,} \\
 1 & \text{if $n' \, \equiv \, 3(\operatorname{mod} \, 4)$}. 
 \end{cases}
\]
 Being consistent with the work of \cite{MadillRampersad2013}, 
 we then write $\text{{\bf f}} = (f_{n})_{n \geq 1}$ as an infinite word, yielding
\begin{equation}\label{numericalf}
 \text{{\bf f}} = 001001100011011000100111001101100010011000110111001001\ldots
\end{equation}
 with the consecutive terms of the word displayed in \eqref{numericalf}
 being indexed in the On-line Encyclopedia of Integer Sequences as {\tt A014707}
 and forming the (regular) \emph{paperfolding sequence}. 

 The subword complexity of $\text{{\bf f}}$ is such that 
\begin{equation}\label{numericalpf}
 \left( \rho_{\text{{\bf f}}}(n) : n \in \mathbb{N} \right) 
 = \big( 2, 4, 8, 12, 18, 23, 28, 32, 36, 40, 44, 48, 52, 56, \ldots \big), 
\end{equation}
 with the integer sequence in \eqref{numericalpf} being indexed in the OEIS as 
 {\tt A337120}. It was first proved by \cite{Allouche1992} in 1992 
 that $ \rho_{\text{{\bf f}}}$ satisfies 
\begin{equation}\label{displayAllouche}
 \rho_{\text{{\bf f}}}(n) = 4n \ \ \ \text{for all $n \geq 7$}.
\end{equation}
 In contrast to the sequence of values in \eqref{numericalpf}, 
 we have that 
\begin{equation}\label{numericalpf_ab}
 \left( \rho_{\text{{\bf f}}}^{\operatorname{ab}}(n) : n \in \mathbb{N} \right) 
 = \big( 2, 3, 4, 3, 4, 5, 4, 3, 4, 5, 6, 5, 4, 5, 4, 3, 4, 5, 6, 5, \ldots \big), 
\end{equation}
 with the integer sequence in \eqref{numericalpf_ab} being indexed in the OEIS as {\tt A214613} 
 and having been shown to be $2$-regular by \cite{MadillRampersad2013}. 
 Informally, the contrast between the respective terms of 
 \eqref{numericalpf} and \eqref{numericalpf_ab} 
 suggests a range of possibilities ``between'' 
 the subwords counted by 
 $ \rho_{\text{{\bf f}}}(n)$ and $ \rho_{\text{{\bf f}}}^{\operatorname{ab}}(n)$ respectively, for a 
 positive integer $n$, and this raises questions as to how
 Definition \ref{definektort} can be used to formalize this idea. 
 In this direction, computer calculation shows that
\begin{equation}\label{numericaltort}
 \left( \rho^{\textsf{t}}_{\text{{\bf f}}}(n) : n \in \mathbb{N} \right) 
 = \big( 2, 3, 5, 7, 12, 18, 26, 32, 36, 40, 44, 48, 52, 56, 60, \ldots \big), 
\end{equation}
 noting the contrast to both \eqref{numericalpf} and \eqref{numericalpf_ab}. 
 The integer sequence in \eqref{numericaltort} is not currently indexed in the OEIS, 
 suggesting that Definition \ref{definektort} is new. 
 To illustrate the contrast between 
 $\rho_{\text{{\bf f}}}$ and $\rho_{\text{{\bf f}}}^{\textsf{t}}$, 
 consider, for example, the valuation 
 $\rho_{\text{{\bf f}}}(2) = 4$, 
 which gives that all four length-2 binary subwords appear as subwords of $\text{{\bf f}}$. 
 In contrast, we see that 
$\rho^{\textsf{t}}_{\text{{\bf f}}}(2) = 3$, since $ 10 \sim_{\textsf{t}} 01$.

The main aim of this section is to prove the following (Theorem~\ref{evalpapertort}):
\begin{equation}\label{purported}
 \rho_{\text{{\bf f}}}^{\textsf{t}}(n) = \rho_{\text{{\bf f}}}(n) \ \ \ \text{for $n \geq 8$}. 
\end{equation}
 The relation in \eqref{purported} provides a natural companion to Allouche's evaluation for $\rho_{\text{{\bf f}}}$ in 
 \eqref{displayAllouche}, and suggests that the way the $\textsf{tortoise}$ operator
 acts on binary words
 together with the definition of tortoise complexity
 help to give light to structural aspects about paperfolding subwords. 
 For example, one might consider the desired formula in \eqref{purported}
 in relation to the work of \cite{AlloucheBousquetMelou1994}, 
 which concerned relations as in 
 $ u_{n} = u_{n + r 2^{i+1}} $ for $n = 2^{j}(2k+1)$ and $i > j$ and for $r \geq 0$
 and for generalized paperfolding sequences $u$. 

 The application of $\textsf{tortoise}$ to a given binary word may be seen as natural in the context of combinatorics on words and related 
 areas in computer science, since, as above, this has the effect of repositioning an initial $1$-character (assuming there is such a character) 
 at the end, providing an extension of the usual lexicographic sorting operation
 in the manner indicated in \eqref{tortsort}. 
 The purported relation in \eqref{purported}
 is helpful in terms of developing a 
 deeper understanding of how the factor and abelian complexity functions
 for $\text{{\bf f}}$ relate to one another, since, as suggested previously, 
 the complexity function $\rho_{\text{{\bf f}}}^{\textsf{t}}$ may be seen as ``between''
 the extremes given by $\rho_{\text{{\bf f}}}$ and $\rho_{\text{{\bf f}}}^{\operatorname{ab}}$. 
 This together with the evaluation of 
 $\rho_{\text{{\bf f}}}$ by \cite{Allouche1992} 
 and the \cite{MadillRampersad2013} recursion for $\rho_{\text{{\bf f}}}^{\operatorname{ab}}$ 
 motivate our below proof of \eqref{purported}
 using the {\tt Walnut} software, with reference to the text by \cite{Shallit2023logical} on this 
 software. 
 
\begin{theorem}\label{keyWalnut}
 For subwords $w$ and $v$ of $\text{{\bf f}}$ such that $|w| = |v| \geq 8$ (noting that each of these words necessarily contains a 
 0-character and a 1-character), 
 if the word obtained from $w$ by repositioning the first $1$ in $w$ at the end
 equals the word obtained from $v$ by repositioning the first $1$ in $v$ at the end, 
 then $v = w$. 
\end{theorem}

\begin{proof}
We use the {\tt Walnut} system. In {\tt Walnut} the paperfolding word ${\bf f}$ is built-in
under the name {\tt P} and is indexed starting from $0$, so that {\tt P[n]} refers
to $f_{n+1}$. In {\tt Walnut}, we input the following commands. 
\begin{verbatim}
def pffactoreq "At t<n => P[i+t]=P[j+t]":
def pfsecond1 "Es s<r & P[i+s]=@1 & P[i+r]=@1 & 
(At (t<r & t!=s)=>P[i+t]=@0)":
def pftortoise "n>=8 & Er $pfsecond1(i,r) & $pfsecond1(j,r) &
$pffactoreq(i+r,j+r,n-r)":
def pflemma "Ai,j,n $pftortoise(i,j,n) => $pffactoreq(i,j,n)":
\end{verbatim}
The output of these commands is as follows:
\begin{itemize}
\item {\tt pffactoreq} is an automaton that accepts exactly the triples $(i,j,n)$
of natural numbers such that {\tt P[i..i+n-1] == P[j..j+n-1]};

\item {\tt pfsecond1} is an automaton that accepts exactly the pairs $(i,r)$
of nonnegative integers such that $i+r+1$ is the position of the second $1$
starting from position $i+1$ in $\text{{\bf f}}$; and

\item {\tt pftortoise} is an automaton that accepts exactly the triples $(i,j,n)$
of natural numbers such that $n \geq 8$ and {\tt P[i..i+n-1]} and {\tt P[j..j+n-1]}
are tortoise-equivalent (where we can check that both words have the second $1$
at the same position and are identical after that).
\end{itemize}
Finally, {\tt pflemma} is a predicate that checks that for $n\geq 8$, the factors
{\tt P[i..i+n-1]} and {\tt P[j..j+n-1]} are tortoise-equivalent only when they are equal.
This returns {\tt TRUE}, proving the theorem.
\end{proof}

\begin{example}
 The $ \rho_{\text{{\bf f}}}(8) = 32$ length-8 subwords of $\text{{\bf f}}$
 are as below. 

\begin{center}
 00100110 \ \ \ \ \ \ \ 01001100 \ \ \ \ \ \ \ 10011000 \ \ \ \ \ \ \ 00110001 \ \ \ \ \ \ \ 01100011 

11000110 \ \ \ \ \ \ \ 10001101 \ \ \ \ \ \ \ 00011011 \ \ \ \ \ \ \ 00110110 \ \ \ \ \ \ \ 01101100 

 11011000 \ \ \ \ \ \ \ 10110001 \ \ \ \ \ \ \ 01100010 \ \ \ \ \ \ \ 11000100 \ \ \ \ \ \ \ 10001001 
 
 00010011 \ \ \ \ \ \ \ 00100111 \ \ \ \ \ \ \ 01001110 \ \ \ \ \ \ \ 10011100 \ \ \ \ \ \ \ 00111001 

 01110011 \ \ \ \ \ \ \ 11100110 \ \ \ \ \ \ \ 11001101 \ \ \ \ \ \ \ 10011011 \ \ \ \ \ \ \ 00110111 

 01101110 \ \ \ \ \ \ \ 11011100 \ \ \ \ \ \ \ 10111001 \ \ \ \ \ \ \ 01110010 \ \ \ \ \ \ \ 11100100 

 11001001 \ \ \ \ \ \ \ 10010011 
\end{center}

\noindent Now, for each length-8 subword $w$ of $\text{{\bf f}}$ listed above, 
 we compute $\textsf{tortoise}(w)$. The results are listed below (consistent with the above ordering). 
We see that all the words in the table below are distinct, and hence that every length-$8$
$\sim_{\textsf{t}}$-equivalence class of ${\bf f}$ is a singleton, as implied by Theorem~\ref{keyWalnut}.

\begin{center}
 00001101 \ \ \ \ \ \ \ 00011001 \ \ \ \ \ \ \ 00110001 \ \ \ \ \ \ \ 00100011 \ \ \ \ \ \ \ 01000111 

 10001101 \ \ \ \ \ \ \ 00011011 \ \ \ \ \ \ \ 00010111 \ \ \ \ \ \ \ 00101101 \ \ \ \ \ \ \ 01011001 

 10110001 \ \ \ \ \ \ \ 01100011 \ \ \ \ \ \ \ 01000101 \ \ \ \ \ \ \ 10001001 \ \ \ \ \ \ \ 00010011 
 
 00000111 \ \ \ \ \ \ \ 00001111 \ \ \ \ \ \ \ 00011101 \ \ \ \ \ \ \ 00111001 \ \ \ \ \ \ \ 00110011 

 01100111 \ \ \ \ \ \ \ 11001101 \ \ \ \ \ \ \ 10011011 \ \ \ \ \ \ \ 00110111 \ \ \ \ \ \ \ 00101111 

 01011101 \ \ \ \ \ \ \ 10111001 \ \ \ \ \ \ \ 01110011 \ \ \ \ \ \ \ 01100101 \ \ \ \ \ \ \ 11001001 

 10010011 \ \ \ \ \ \ \ 00100111 
\end{center} 
\end{example}

\begin{theorem}\label{evalpapertort}
 The tortoise complexity evaluation $ \rho_{\text{{\bf f}}}^{\textsf{t}}(n) = 4n$ holds for all $n \geq 8$. 
\end{theorem}

\begin{proof}
 By Theorem \ref{keyWalnut}, two subwords $v$ and $w$ of $\text{{\bf f}}$ such that $|v| = |w| \geq 8$ are tortoise-equivalent if 
 and only if $v = w$, so that \eqref{purported} holds and so that the desired evaluation 
 follows from Allouche's evaluation in \eqref{displayAllouche}. 
\end{proof} 

Computationally, 
$$ \rho_{\text{{\bf f}}}^{\textsf{t}^{(2)}}(n) = 4n \ \ \ \text{for all $n \geq 13$}, $$
appears to hold, and more generally, the following appear to hold:
\begin{align*}
 \rho_{\text{{\bf f}}}^{\textsf{t}^{(3)}}(n) & = 4n \ \ \ \text{for all $n \geq 14$}, \\ 
 \rho_{\text{{\bf f}}}^{\textsf{t}^{(4)}}(n) & = 4 n \ \ \ \text{for all $n \geq 25$}, \\ 
 \rho_{\text{{\bf f}}}^{\textsf{t}^{(5)}}(n) & = 4 n \ \ \ \text{for all $n \geq 26$}, \\ 
 & \vdots 
\end{align*}
 For a positive integer $k$, let $s(k)$
be the least positive integer $N$ such that
$\rho_{\text{{\bf f}}}^{\textsf{t}^{(k)}}(n)=4n$ for every $n\geq N$, 
 assuming that such a value exists. 
 For the initial values of $s(k)$ suggested above, 
 the associated integer sequence $(8, 13, 14, 25, 26, \ldots)$ 
 is not currently in the OEIS. 
In the Conclusion, we also consider the problem of evaluating
$\operatorname{abel}_{\text{{\bf f}}}(k)$. 

 \section{Stack sorting and the TM sequence}\label{secTMmain}
 The Thue--Morse sequence is, as above, a prototypical automatic sequence and one of the most fundamental automatic sequences. This 
 may be illustrated with the exposition on this sequence by \cite{AlloucheShallit1999}. For a nonnegative integer $n$, 
 by setting $\text{{\bf t}}_{n}$ as the number of $1$-digits, modulo 2, in the binary expansion of $n$, 
 the Thue--Morse sequence may be defined as the sequence $\left( \text{{\bf t}}_{n} : n \in \mathbb{N}_{0} \right)$, writing
\begin{equation}\label{numericalTM}
 \text{{\bf t}} = 
 011010011001011010010110011010011001011001101001011010\ldots
\end{equation}
 with \eqref{numericalTM} being indexed in the OEIS as {\tt A010060}. 
 As it turns out, the tortoise complexity for $\text{{\bf t}}$ 
 seems to reveal deeper properties relative to the paperfolding sequence, 
 in view of Theorem \ref{mainTM} below, which we apply to evaluate $\rho_{\text{{\bf t}}}^{\textsf{t}}$. 

 The subword complexity of $\text{{\bf t}}$ is such that 
\begin{equation}\label{numericalrhotm}
 \left( \rho_{\text{{\bf t}}}(n) : n \in \mathbb{N} \right) 
 = (2, 4, 6, 10, 12, 16, 20, 22, 24, 28, 32, 36, 40, 42, 44, \ldots), 
\end{equation}
 yielding the sequence {\tt A005942} in the OEIS, 
 and satisfies 
\begin{align*}
 \rho_{\text{{\bf t}}}(2n) & = \rho_{\text{{\bf t}}}(n) + \rho_{\text{{\bf t}}}(n+1), \ \text{and} \\ 
 \rho_{\text{{\bf t}}}(2n+1) & = 2 \rho_{\text{{\bf t}}}(n+1) 
\end{align*}
 for $n$ greater than $1$. The above recurrence relations for $\rho_{\text{{\bf t}}}$ were proved independently by \cite{Brlek1989}, 
 \cite{deLucaVarricchio1989}, and \cite{Avgustinovich1994}. This raises questions as to how the 
 complexity functions $\rho_{\text{{\bf t}}}$ and $\rho_{\text{{\bf t}}}^{\textsf{t}}$ relate to one another, 
 and this provides the main topic of the current section. 
 In contrast to \eqref{numericalrhotm}, we have that the integer sequence 
\begin{equation}\label{abTMtort}
 \left( \rho^{\operatorname{ab}}_{\text{{\bf t}}}(n) : n \in \mathbb{N} \right) 
 = (2, 3, 2, 3, 2, 3, 2, 3, 2, 3, 2, 3, 2, 3, 2, 3, 2, 3, 2, 3, \ldots) 
\end{equation}
 is $2$-periodic, whereas 
\begin{equation}\label{rhott}
 \left( \rho^{\textsf{t}}_{\text{{\bf t}}}(n) : n \in \mathbb{N} \right) 
 = (2, 3, 3, 6, 8, 12, 15, 17, 21, 26, 28, 32, 36, 38, 42, \ldots) 
\end{equation}
 is not currently included in the OEIS. 
 As suggested above, the tortoise complexity function
 $\rho_{\text{{\bf t}}}^{\textsf{t}}$ may be seen as being ``between''
 the two extremes given by the complexity functions in 
 \eqref{numericalrhotm} and \eqref{abTMtort}, 
 and Theorems~\ref{TM_eq_cl} and~\ref{mainTM} below 
 determine the equivalence classes with respect to $\sim_{\textsf{t}}$ underlying this phenomenon. 

\begin{theorem}\label{TM_eq_cl}
 Let $n$ be an integer such that $n \geq 10$. If there are length-$n$ subwords $w \neq w'$
 of $\text{{\bf t}}$ such that $w \sim_{\textsf{t}} w'$, then
 $\{w,w'\} = \{01z, 10z\}$; in particular, $z$ is left special.
\end{theorem}

\begin{proof}
By Lemma~\ref{lem_left_sp}, we have $\{w,w'\} = \{u1vz, uv1z\}$, where $u \in 0^*$ and $v \in 0^+$.
It is well-known that ${\bf t}$ does not contain $000$ (see~\cite{AlloucheShallit1999}).
Hence, we have $uv=0$ or $uv=00$. In the first case, we get $\{w,w'\} = \{01z, 10z\}$,
as required.

If $uv=00$, then there are two possibilities: either $\{w,w'\} = \{001z, 100z\}$ or
$\{w,w'\} = \{001z, 010z\}$. By brute force computer check, we verify that if $|z|\geq 7$,
the word ${\bf t}$ does not contain a pair of factors of the form $001z$ and $100z$.
A similar check shows that ${\bf t}$ does not contain a pair $001z$ and $010z$ either.
(Note that the minimal length $7$ for $z$ is optimal here since ${\bf t}$ does
contain the two factors $001.011001$ and $010.011001$.)
\end{proof}

The next result gives a somewhat more explicit version of the previous
theorem, in the sense that whenever we have a $\sim_{\textsf{t}}$ equivalence class
of size $2$, i.e., $\{v,w\} = \{01y, 10y\}$, the theorem specifies where precisely
we can find an occurrence of either $v$ or $w$ in ${\bf t}$.
Note the two cases in the statement of the theorem and compare with Lemma~\ref{lem_dLV}.

\begin{theorem}\label{mainTM}
 Let $n$ be an integer such that $n \geq 10$. Let the binary expansion of $n-3$ be written as
\begin{equation}\label{nminus3}
 n - 3 = [b_1 b_2 \cdots b_{m} ]_{2}. 
\end{equation}
 Define 
\begin{align}
 \operatorname{counterex}_1(n) & = \text{{\bf t}}_{2^{m+1}-2} \text{{\bf t}}_{2^{m+1}-1} 
 \cdots \text{{\bf t}}_{2^{m+1} + n -3}, \label{counter1} \\ 
 \operatorname{counterex}_2(n) & = \text{{\bf t}}_{2^{m+1} +2^{m-1} - 2 } \text{{\bf t}}_{2^{m+1} +2^{m-1} - 1 } 
 \cdots \text{{\bf t}}_{2^{m+1} +2^{m-1} +n - 3 }, \label{counter2} \\ 
 \operatorname{counterex}_3(n) & = \text{{\bf t}}_{2^{m+1} + 2^{m-1} + 2^{m-2} - 2 } 
 \cdots \text{{\bf t}}_{2^{m+1} + 2^{m-1} + 2^{m-2} + n - 3}, \ \text{and} \label{counter3} \\ 
 \operatorname{counterex}_4(n) & = \text{{\bf t}}_{2^{m+1} + 2^{m } - 2 } 
 \text{{\bf t}}_{2^{m+1} + 2^{m } - 1 } 
 \cdots \text{{\bf t}}_{2^{m+1} + 2^{m } +n-3}. \label{counter4}
\end{align}
 Let $w$ be a length-$n$ subword of $\text{{\bf t}}$. 
 If $b_1 b_2 = 10$, and 
 if there is a subword $v \neq w$
 of $\text{{\bf t}}$ such that $v \sim_{\emph{\textsf{t}}} w$, 
 then $w$ or $v$ is equal to one of the counterexamples among \eqref{counter1}--\eqref{counter4}. If $b_1 b_2 = 11$, and if there is a 
 subword $v \neq w$ of $\text{{\bf t}}$ such that $v \sim_{\emph{\textsf{t}}} w$, 
 then $w$ or $v$ is equal to one of the counterexamples among 
 \eqref{counter1} and \eqref{counter4}. Furthermore, the words
 given by \eqref{counter1}--\eqref{counter4} all belong to different
 $\sim_{\emph{\textsf{t}}}$-equivalence classes.
\end{theorem}

\begin{proof}
 We again use the {\tt Walnut} system. Note that the Thue--Morse word is built-in
 to {\tt Walnut} and has the name {\tt T}. We apply the following sequence of commands. 
\begin{verbatim}
def tmfactoreq "At t<n => T[i+t]=T[j+t]":
def tmsecond1 "Es s<r & T[i+s]=@1 & T[i+r]=@1 & 
(At (t<r & t!=s)=>T[i+t]=@0)":
def tmtortoise "n>=5 & Er $tmsecond1(i,r) & $tmsecond1(j,r) &
$tmfactoreq(i+r,j+r,n-r)":
reg power2 msd_2 "0*10*":
reg starts10 msd_2 "0*10(0|1)*":
reg starts11 msd_2 "0*11(0|1)*":
def pow2forcounter "n>=10 & $power2(i) & 2*i+3<=n & 4*i+3>n":
def counter1 "n>=10 & Ei $pow2forcounter(i,n) & j = 8*i-2":
def counter2 "n>=10 & Ei $pow2forcounter(i,n) & j = 10*i-2":
def counter3 "n>=10 & Ei $pow2forcounter(i,n) & j = 11*i-2":
def counter4 "n>=10 & Ei $pow2forcounter(i,n) & j = 12*i-2":
def counterany "$counter1(j,n)|$counter2(j,n)|$counter3(j,n)|
$counter4(j,n)":
def counter1or4 "$counter1(j,n)|$counter4(j,n)":
def case10 "Ai,j,n (n>=10 & $starts10(n-3) & 
$tmtortoise(i,j,n) &
~$tmfactoreq(i,j,n)) => (Er $counterany(r,n) &
($tmfactoreq(i,r,n)|$tmfactoreq(j,r,n)))":
def case11 "Ai,j,n (n>=10 & $starts11(n-3) & 
$tmtortoise(i,j,n) &
~$tmfactoreq(i,j,n)) => (Er $counter1or4(r,n) &
($tmfactoreq(i,r,n)|$tmfactoreq(j,r,n)))":
def distinctcounters "Ai,j,n (n>=10 & $counterany(i,n) &
$counterany(j,n) & $tmtortoise(i,j,n)) => i=j":
\end{verbatim}
The results of these commands are as follows:
\begin{itemize}

\item {\tt tmfactoreq} is an automaton that accepts exactly the triples
$(i,j,n)$ such that {\tt T[i..i+n-1] == T[j..j+n-1]};

\item {\tt tmsecond1} is an automaton that accepts exactly the pairs
$(i,r)$ such that $i+r$ is the position of the second $1$
starting from position $i$ in $\text{{\bf t}}$;
\item {\tt tmtortoise} is an automaton that accepts exactly the triples
$(i,j,n)$ such that $n \geq 5$ and 
 {\tt T[i..i+n-1]} and {\tt T[j..j+n-1]} 
are tortoise-equivalent (we check that both words have the second $1$
at the same position and are identical after that, noting that the condition $n \geq 5$ guarantees
there are two $1$'s); 
\item {\tt power2(n)} evaluates to {\tt TRUE} exactly when the base-2 expansion
of $n$ contains exactly one $1$; i.e., when $n$ is a power of $2$;
\item {\tt starts10(n)} evaluates to {\tt TRUE} exactly when
the base-2 expansion of $n$ starts with 10;
\item {\tt starts11(n)} evaluates to {\tt TRUE} exactly when
the base-2 expansion of $n$ starts with 11;
\item {\tt pow2forcounter} is an automaton 
 that accepts exactly the pairs
$(i,n)$ such that $(i,n) = (2^{m-2}, n)$, where $m$ is the
length of the base-2 expansion of $n-3$, 
recalling the notation in \eqref{nminus3};
\end{itemize}
From the above definitions, we have that 
\begin{itemize}
\item {\tt counter1} is an automaton that accepts exactly the pairs
$(j,n)$ such that $j$ is the starting position of $\operatorname{counterex}_1(n)$, 
 and similarly for the remaining functions among \eqref{counter1}--\eqref{counter4}.
\item {\tt counterany} is an automaton that accepts exactly the pairs
$(j,n)$ such that $j$ is the starting position of any
 expression among \eqref{counter1}--\eqref{counter4}. 
\item {\tt counter1or4} is an automaton that accepts exactly the pairs
$(j,n)$ such that $j$ is the starting position of
 $\operatorname{counterex}_1(n)$ or $\operatorname{counterex}_4(n)$.
\end{itemize}
The final commands work as follows:
\begin{itemize}
\item {\tt case10} is a predicate
 that verifies that when $b_1b_2=10$, one of $v$ and $w$ is equal to one of
 the counterexamples among \eqref{counter1}--\eqref{counter4}.
 It evaluates to {\tt TRUE} confirming the claim in this case.
\item {\tt case11} is a predicate
 that verifies that when $b_1b_2=11$, one of $v$ and $w$ is equal to one of
 the counterexamples among \eqref{counter1} and \eqref{counter4}.
 It evaluates to {\tt TRUE} confirming the claim in this case.
\item {\tt distinctcounters} is a predicate that verifies that
the words given by \eqref{counter1}--\eqref{counter4} are pairwise
$\sim_{\textsf{t}}$-inequivalent.
It evaluates to {\tt TRUE}.
\end{itemize}
\end{proof}

\begin{example}
 Let $n = 58$, with $n - 3 = [110111]_2$, with $m = 6$ according to the notation in 
 \eqref{nminus3}. Since $b_1b_2 = 11$, 
 there should be two distinct $2$-sets each 
 consisting of length-58 subwords $w$ and $v$ of $\text{{\bf t}}$ 
 such that $w \neq v$ and $v \sim_{\textsf{t}} w$ 
 and such that no member of one such $2$-set is tortoise-equivalent to any member of the other $2$-set. 
 In this direction, 
 we let 
\begin{equation}\label{TMw1}
 w^{(1)} = \text{{\bf t}}_{126} \text{{\bf t}}_{127} \cdots \text{{\bf t}}_{183}, 
\end{equation}
 with 
 $$w^{(1)} = 0110010110011010010110100110010110011010011001011010010110. $$
 We then let 
\begin{equation}\label{TMw2}
 w^{(2)} = \text{{\bf t}}_{190} \text{{\bf t}}_{191} \cdots \text{{\bf t}}_{247}, 
\end{equation}
 with 
$$
 w^{(2)} = 0101101001100101101001011001101001100101100110100101101001, 
$$
 noting that the Thue--Morse subwords defined in 
 \eqref{TMw1} and \eqref{TMw2} are distinct 
 and are not tortoise-equivalent. We then set 
\begin{equation*}
 v^{(1)} = \text{{\bf t}}_{62} \text{{\bf t}}_{63} \cdots \text{{\bf t}}_{119}, 
\end{equation*}
 as a companion to \eqref{TMw1}, with 
$$ v^{(1)} = 1010010110011010010110100110010110011010011001011010010110, $$
 noting that $w^{(1)} \neq v^{(1)}$ and that $w^{(1)} \sim_{\textsf{t}} v^{(1)}$. 
 Similarly, we set 
\begin{equation*}
 v^{(2)} = \text{{\bf t}}_{94} \text{{\bf t}}_{95} \cdots \text{{\bf t}}_{151}, 
\end{equation*}
 as a companion to \eqref{TMw2}, 
 with 
$$ v^{(2)} = 1001101001100101101001011001101001100101100110100101101001, $$
 noting that 
 $w^{(2)} \neq v^{(2)}$ and that $w^{(2)} \sim_{\textsf{t}} v^{(2)}$. 
 It may be verified that out of 
 the $\rho_{\text{{\bf t}}}(58) = 178$ length-58 subwords of $\text{{\bf t}}$, 
 if any such subword $x$ is such that there exists a subword $y$ of $\text{{\bf t}}$
 such that $x \neq y$ and $x \sim_{\textsf{t}} y$, 
 then $x$ and $y$ are both in $\{ w^{(1)}, w^{(2)}, v^{(1)}, v^{(2)} \}$. 
 Moreover for $m$ as above, we see that the first $\text{{\bf t}}$-index on the right of 
 \eqref{TMw1} is $2^{m+1} -2$, so that 
 $w^{(1)} = \operatorname{counterex}_1(n)$, 
 and the first $\text{{\bf t}}$-index on the right of \eqref{TMw2} 
 is $2^{m+1} + 2^m -2$, 
 so that 
 $w^{(2)} = \operatorname{counterex}_4(n)$, and this agrees with 
 Theorem \ref{mainTM}. 
\end{example}

\begin{theorem}\label{evalTMtort}
 For all $n \geq 10$, we have that 
\begin{displaymath}
 \rho_{\text{{\bf t}}}^{\emph{\textsf{t}}}(n) = \begin{cases} 
 \rho_{\text{{\bf t}}}(n) - 4 & \text{if the first two binary digits of $n-3$ are $1$ and $0$, and} \\
 \rho_{\text{{\bf t}}}(n) - 2 & \text{otherwise.} 
 \end{cases}
\end{displaymath}
\end{theorem}

\begin{proof}
Theorem~\ref{TM_eq_cl} implies that for $n \geq 10$ each length-$n$ $\sim_{\textsf{t}}$-equivalence class contains
at most $2$ elements. Theorem~\ref{mainTM} implies that when $n \geq 10$ and the base-2
expansion of $n-3$ starts with $10$ (resp.~$11$) there are at most four (resp.~two) length-$n$ $\sim_{\textsf{t}}$-equivalence classes that contain
exactly $2$ elements. To show that there are exactly four (resp.~two) such equivalence classes in the first (resp.~second)
case, note first that Theorem~\ref{mainTM} implies that if these equivalence classes all occur, they are all
distinct.

Conversely, let $z$ be a left special factor of ${\bf t}$ of
length $n-2\geq8$. Recall that ${\bf t}$ is composed of $01$
and $10$ blocks beginning at even positions and contains neither
$01010$ nor $10101$ (see \cite{AlloucheShallit1999}).
Thus $z$ contains $00$ or $11$, whose occurrences in ${\bf t}$
start at odd positions, so all occurrences of $z$ have the same
starting parity. This parity cannot be odd, since then the
preceding letter would always be the complement of the first
letter of $z$, contradicting left specialness.
Consequently, occurrences of $0z$ and $1z$ extend to $10z$
and $01z$, respectively. These two factors have the same tortoise
image $0z1$ and, by Theorem~\ref{TM_eq_cl}, form a two-element
equivalence class. Distinct choices of $z$ give distinct images
$0z1$, and hence distinct classes.

 Now Lemma~\ref{lem_dLV} shows
that there are four such $z$ in the first case and $2$ such $z$ in the second case. (Alternatively,
we could use {\tt Walnut} to show that all such equivalence classes occur.)
\end{proof}

 Similar relations can be obtained for $\rho_{\text{{\bf t}}}^{\textsf{t}^{(k)}}$, and we encourage explorations of this, 
 along with the topics for future research covered in Section \ref{secConclusion} below. 

\section{Conclusion}\label{secConclusion}
 We have disregarded the enumeration of subwords up to equivalence by Defant and Kravitz's $\textsf{hare}$ operation, i.e., since we 
 have been dealing with binary sequences and since \eqref{equivbin} holds for a given binary word $b$. We encourage a full 
 exploration of hare-equivalence (by analogy with Definition \ref{ktorteq}) for non-binary sequences. 

 Theorems \ref{evalpapertort} and \ref{evalTMtort} raise questions as to the relationships between the sequences 
 $\big( \rho_{\text{{\bf x}}}^{\textsf{t}}(n) : n \in \mathbb{N}_{0} \big)$ and $\big( \rho_{\text{{\bf x}}}(n) : n \in \mathbb{N}_{0} \big)$
 for a given sequence $\text{{\bf x}}$ over a finite alphabet, 
 especially for the case whereby $\text{{\bf x}}$ is automatic. We encourage a full exploration of this. 
 How can $\rho_{\text{{\bf x}}}^{\textsf{t}}(n)$ be expressed in terms of $\rho_{\text{{\bf x}}}(n)$, 
 for a given automatic sequence $\text{{\bf x}}$ and a given argument $n \in \mathbb{N}_{0}$? 

 From the relation in \eqref{tortsort}, we have that expressions of the form 
\begin{equation}\label{nearlyab}
\textsf{tortoise}^{\left(\max\{\langle w\rangle_{\textsf{tortoise}}-1,0\}\right)}(w),
\end{equation}
 are as close as possible to being sorted, with respect to the $\textsf{tortoise}$ operation. This motivates a full study of ``nearly abelian'' 
 complexity functions, whereby subwords are enumerated up to an equivalence relation defined via \eqref{nearlyab}, by analogy with 
 tortoise complexity functions. What are the nearly abelian complexity functions for $\text{{\bf t}}$ and $\text{{\bf f}}$? 
 
 What would be appropriate as generalizations or variants based on tortoise complexity functions of Sturmian sequences? 
 
 Returning to \eqref{desireabel}, the sequence of expressions of the form $\operatorname{abel}_{\text{{\bf f}}}(k)$ for $k = 1, 2, \ldots$ 
 begins with $$ 2, 3, 5, 6, 9, 10, 11, 14, 17, 18, 19, 21, 22, 25, 27, \ldots $$ and this does not currently agree with any OEIS sequences. 
 How can this sequence be evaluated? How can a recursion for the sequence be determined? What can be said in regard to the growth 
 properties or first differences of this sequence? How can the sequence of arguments $k$ such that 
 $\operatorname{abel}_{\text{{\bf f}}}(k+1) - \operatorname{abel}_{\text{{\bf f}}}(k) = 1$ be evaluated? One may consider the same 
 questions applied to $\operatorname{abel}_{\text{{\bf t}}}$. 

\subsection*{Acknowledgements}
 The authors thank Jean-Paul Allouche for useful discussions related to this paper. 
 The authors are also very grateful to the referees for their careful reading 
 of the paper and for suggesting many important improvements. 

 \subsection*{AI Disclosure}
The first author used GPT-6 Astra Pro for editorial assistance,
mathematical checking, and assistance with revisions to proofs.
This assistance informed corrections to examples and definitions
and materially affected the left-special-factor
argument in the proof of Theorem~\ref{evalTMtort}.
The authors take full responsibility for the mathematical content
and presentation of the manuscript.

\nocite{*}
\bibliographystyle{abbrvnat}
\bibliography{references}
\label{sec:biblio}

\end{document}